\documentclass{amsart}
\usepackage{amsfonts}
\usepackage{amsmath}
\usepackage{amssymb}
\usepackage{amsthm}
\usepackage{graphicx}
\usepackage{capt-of}
\usepackage{tikz}
\usetikzlibrary{patterns}
\usetikzlibrary{decorations.markings}
\usepackage{pgfplots}
\usepackage{tikz-3dplot}
\tdplotsetmaincoords{60}{115}
\pgfplotsset{compat=newest}
\usepackage{multirow}
\usepackage{array} 
\newcolumntype{C}{>{$}c<{$}}

\newtheorem{theorem}{Theorem}

\newtheorem{lemma}{Lemma}

\newtheorem{exmp}{Example}

\begin{document}
\title{$Z$-knotted and $Z$-homogeneous triangulations of surfaces}
\author{Adam Tyc}
\subjclass[2000]{}

\address{Adam Tyc: Institute of Mathematics, Polish Academy of Sciences, \'Sniadeckich 8, 00-656 Warszawa, Poland}
\email{atyc@impan.pl}

\maketitle
\begin{abstract}
A triangulation is called $z$-knotted if it has a single zigzag (up to reversing). 
A $z$-orientation on a triangulation is a minimal collection of zigzags which double covers the set of edges. 
An edge is of type I if zigzags from the $z$-orientation pass through it in different directions, otherwise this edge is of type II. 
If all zigzags from the $z$-orientation contain precisely two edges of type I after any edge of type II, then the $z$-oriented triangulation is said to be $z$-homogeneous. 
We describe an algorithm transferring each $z$-homogeneous trianguation to other $z$-homogeneous triangulation which is also $z$-knotted. 
\end{abstract}

\section{Introduction}
{\it Zigzags} \cite{DDS-book, Lins1} (also known as {\it Petrie walks} \cite{Coxeter} and {\it closed left-right paths} \cite{GR-book, Shank}) are sequences of edges in a graph embedded in a surface such that any two consecutive edges are different, have a common vertex and are contained in a certain face, but for any three consecutive edges there is no face containing them. 
Zigzags are successfully used in various branches of mathematics. 
They can be realized as the projections of links in knot theory \cite{DDS-book}. 
In chemical graph theory, zigzags in graphs embedded into a sphere are applied to enumerating all combinatorial possibilities for fullerenes \cite{BD, DDS-book}. 
Embedded graphs which contain a single zigzag, i.e. {\it $z$-knotted} graphs, are closely connected to {\it Gauss code problem} \cite{CrRos, GR-book, Lins2}. 
In fact, zigzags of such embedded graphs are geometrical realizations of Gauss codes. 
Infinite series of $z$-knotted triangulations for any closed (not necessarily) oriented surfaces were constructed in \cite{PT1}. 

The concept of {\it $z$-monodromy} for triangle faces is introduced in \cite{PT2}. 
The $z$-monodromy $M_F$ of a face $F$ is a permutation of six oriented edges of this face. 
For two consecutive oriented edges $e_0$ and $e$ in $F$ we consider the unique zigzag containing the sequence $e_0, e$ and define $M_F(e)$ as the first oriented edge of $F$ which occurs in this zigzag after $e$. 
There are precisely $7$ types of $z$-monodromies and each of these types is realized. 
The $z$-monodromy is used to prove that any triangulation $\Gamma$ of an arbitrary (not necessarily oriented) closed surface admits a {\it $z$-knotted shredding}, i.e. $\Gamma$ can be modified to a $z$-knotted triangulation by triangulating several of its faces \cite[Theorem 3.1]{PT2}. 
Some combinatorial properties of $z$-monodromies can be found in \cite{PT3}. 

A {\it $z$-orientation} is a collection of zigzags which double covers the set of edges in an embedded graph \cite{DDS-book}. 
For the $z$-knotted case, there is a unique $z$-orientation (up to reversing). 
In general, there are two possibilities for an edge: zigzags from the $z$-orientation pass through the edge in different directions or twice in the same direction. 
We say that the edge is {\it of type} I in the first case or it is {\it of type} II in the second case. 
For every face in a triangulation we have the following possibilities: 
two edges are of type I and the third is of type II or all edges are of type II.
If the first case is realized for all faces, then the number of edges of type I is the twofold number of edges of type II and we say that a $z$-oriented triangulation is {\it $z$-homogeneous} if all zigzags from the $z$-orientation contain precisely two edges of type I after any edge of type II. 
Such triangulations are important for the following reason: 
there is a one-to-one correspondence between $z$-homogeneous triangulations and embeddings of simple Eulerian digraphs into surfaces \cite{T1}. 
However, the $z$-homogenity is related to the fixing of $z$-orientation. 
The uniqueness of $z$-orientation (up to reversing) is equivalent to the $z$-knottedness. 
So, it will be significant to obtain a $z$-knotted triangulation, i.e. a triangulation with a unique $z$-orientation up to reversing, which is $z$-homogeneous. 

The main result of this paper states that every $z$-homogeneous triangulation can be modified in a $z$-homogeneous triangulation which is also $z$-knotted. 
Our key tool is the $z$-monodromy for a pair of consecutive edges of type II. 
There are precisely $13$ classes of such $z$-monodromies and each of them is realized (the realization of some of these classes is established by computer).
In our construction, we use the connected sum of $z$-homogeneous triangulations by such special pairs of edges. 
We take two $z$-homogeneous triangulations and cut each of them by a special pair of edges. 
The gluing gives a $z$-homogeneous triangulation of the connected sum of the corresponding surfaces. 
Now, we take an arbitrary $z$-homogeneous triangulation $\Gamma$. 
Using properties of $z$-monodromies for pairs of edges, we obtain a $z$-homogeneous triangulation (of the same surface) by gluing to $\Gamma$ some $z$-homogeneous triangulations of spheres and show that this new $z$-homogeneous triangulation is $z$-knotted. 

\section{Zigzags and $z$-orientations of embedded graphs}

Let $M$ be a connected closed $2$-dimensional (not necessarily orientable) surface and let $\Gamma$ be a closed $2$-cell embedding of a connected simple finite graph in $M$. 
We assume that this graph does not contain vertices of degree $2$. 
Two distinct edges are {\it adjacent} if there is a face containing them and the edges have a common vertex. 

A {\it zigzag} in $\Gamma$ is a sequence of edges $\{e_i\}_{i\in\mathbb{N}}$ satisfying the following conditions for every $i\in\mathbb{N}$:
\begin{enumerate}
\item[$\bullet$] $e_i$ and $e_{i+1}$ are adjacent,
\item[$\bullet$] the faces containing $e_i, e_{i+1}$ and $e_{i+1}, e_{i+2}$ are distinct and $e_i, e_{i+2}$ do not have a common vertex.
\end{enumerate}
Since $\Gamma$ is finite, for every zigzag $Z=\{e_{i}\}_{i\in {\mathbb N}}$
there is a natural number $n>0$ such that $e_{i+n}=e_{i}$ for every $i\in {\mathbb N}$, i.e. $Z$ is a cyclic sequence. 
For any pair of adjacent edges there is a zigzag containing them. 
Conversely, any zigzag is completely determined by a pair of consecutive edges belonging to it.
If $Z=\{e_1,\dots e_n\}$ is a zigzag, then the squence $Z^{-1}=\{e_n,\dots e_1\}$ also is a zigzag which will be called {\it reversed}. 
A zigzag cannot be self-reversed, i.e. $Z\neq Z^{-1}$ for any zigzag $Z$ (see \cite{PT2}).
We say that $\Gamma$ is {\it $z$-knotted} if it contains precisely one pair of zigzags, in other words, there is a single zigzag up to reversing. 

A $z$-{\it orientation} $\tau$ of $\Gamma$ is a collection of zigzgas such that for each zigzag $Z$ we have $Z\in \tau$ or $Z^{-1}\in \tau$. 
If $\Gamma$ contains precisely $k$ zigzags up to reversing, then there are $2^k$ distinct $z$-orientations of $\Gamma$.
For every $z$-orientation $\tau=\{Z_{1},\dots,Z_{k}\}$ the $z$-orientation $\tau^{-1}=\{Z^{-1}_{1},\dots, Z^{-1}_{k}\}$ will be called {\it reversed} to $\tau$. 
The embedded graph $\Gamma$ with a $z$-orientation $\tau$ is denoted by $(\Gamma, \tau)$ and called a {\it $z$-oriented embedded graph}.

Let $\tau$ be a $z$-orientation of $\Gamma$. 
For every edge $e$ of $\Gamma$ one of the following possibilities is realized:
\begin{enumerate}
\item[$\bullet$] there is a zigzag $Z\in \tau$ such that $e$ occurs in this zigzag twice and other zigzags from $\tau$
do not contain $e$,
\item[$\bullet$] there are two distinct zigzags $Z,Z'\in\tau$ such that $e$ occurs in each of these zigzags only ones 
and other zigzags from $\tau$ do not contain $e$. 
\end{enumerate}
In the first case, if $Z$ passes through $e$ twice in the opposite or in the same direction, then $e$ is an edge {\it of type} I or {\it of type} II, respectively. 
Similarly, in the second case: $e$ is an edge {\it of type} I or {\it of type} II if $Z$ and $Z'$ pass through $e$ in different directions or in the same direction, respectively. 
Thus, $\tau$ induces a direction of any edge of type II.
The reversing of a $z$-orientation does not change types of edges, but it reverses directions of edges of type II.

By \cite[Proposition 1]{T1}, if a face of $(\Gamma, \tau)$ is a triangle, then one of the following possibilities is realized:
\begin{enumerate}
\item[{\rm (I)}] the face contains two edges of type {\rm I} and the third edge is of type {\rm II}, see {\rm Fig.1(a)};
\item[{\rm (II)}] all edges of the face are of type {\rm II} and form a directed cycle, see {\rm Fig.1(b)}.
\end{enumerate}

\begin{center}
\begin{tikzpicture}[scale=0.6]

\draw[fill=black] (0,2) circle (3pt);
\draw[fill=black] (-1.7320508076,-1) circle (3pt);
\draw[fill=black] (1.7320508076,-1) circle (3pt);

\draw [thick, decoration={markings,
mark=at position 0.62 with {\arrow[scale=1.5,>=stealth]{>>}}},
postaction={decorate}] (5.1961524228,2) -- (3.4641016152,-1);

\draw [thick, decoration={markings,
mark=at position 0.62 with {\arrow[scale=1.5,>=stealth]{>>}}},
postaction={decorate}] (3.4641016152,-1) -- (6.9282032304,-1);

\draw [thick, decoration={markings,
mark=at position 0.62 with {\arrow[scale=1.5,>=stealth]{<<}}},
postaction={decorate}] (5.1961524228,2) -- (6.9282032304,-1);

\node at (0,-1.65) {(a)};

\draw[fill=black] (5.1961524228,2) circle (3pt);
\draw[fill=black] (3.4641016152,-1) circle (3pt);
\draw[fill=black] (6.9282032304,-1) circle (3pt);

\draw [thick, decoration={markings,
mark=at position 0.62 with {\arrow[scale=1.5,>=stealth]{><}}},
postaction={decorate}] (0,2) -- (-1.7320508076,-1);

\draw [thick, decoration={markings,
mark=at position 0.62 with {\arrow[scale=1.5,>=stealth]{>>}}},
postaction={decorate}] (-1.7320508076,-1) -- (1.7320508076,-1);

\draw [thick, decoration={markings,
mark=at position 0.62 with {\arrow[scale=1.5,>=stealth]{><}}},
postaction={decorate}] (0,2) -- (1.7320508076,-1);

\node at (5.1961524228,-1.65) {(b)};
\end{tikzpicture}
\captionof{figure}{ }
\end{center}
This face is said to be {\it of type} I or {\it of type} II if the corresponding possibility is realized. 
The reversing of the $z$-orientation $\tau$ does not change types of faces. 

From this moment, we assume that $(\Gamma, \tau)$ is a $z$-oriented triangulation, i.e. each face is a triangle. 
If all faces are of type I, then the number of edges of type I is the twofold number of edges of type II. 
In this case, a zigzag of $\Gamma$ is said to be {\it homogeneous} if it is a cyclic sequence $\{e_{i},e'_{i},e''_{i}\}^{n}_{i=1}$,
where each $e_{i}$ is an edge of type II and all $e'_{i},e''_{i}$ are edges of type I. 
If all zigzags of $(\Gamma, \tau)$ are homogeneous, then we say that $(\Gamma, \tau)$ is {\it $z$-homogeneous}. 
If $(\Gamma, \tau)$ is $z$-homogeneous, then the same holds for $(\Gamma, \tau^{-1})$. 

Denote by $\Gamma_{II}$ the subgraph of $\Gamma$ formed by all edges of type II and all vertices on these edges. 
Following \cite[Proposition 3]{T1}, we describe a one-to-one correspondence between $z$-homogeneous triangulations and embeddings of simple Eulerian digraphs into surfaces. 
If $(\Gamma,\tau)$ is a $z$-homogeneous triangulation, then $\Gamma_{II}$ is a closed $2$-cell embedding of a simple Eulerian digraph such that every face is a directed cycle. 
Conversely, let us consider closed $2$-cell embedding of a simple Eulerian digraph $\Gamma'$ such that every face is a directed cycle. 
Denote by ${\rm T}(\Gamma')$ a triangulation obtained in the following way: 
in the interior of any face of $\Gamma'$ we add a vertex and join it with each vertex of this face by an edge. 
Then there exists a unique $z$-orientation $\tau$ such that $({\rm T}(\Gamma'),\tau)$ is $z$-homogeneous and the set of its edges of type II is precisely the set of all edges of $\Gamma'$ (with the same orientations). 

\begin{exmp}\label{ex1}{\rm
The $n$-gonal bipyramid $BP_n$, $n\geq 3$, is a triangulation of $\mathbb{S}^2$ containing an $n$-gone with consecutive vertices $v_1,\dots,v_n$ and two disjoint vertices $a, b$ connected with all vertices of the $n$-gone (see Fig. 2 for $n=3$). 
\begin{center}
\begin{tikzpicture}[scale=0.4]

\coordinate (A) at (0.5,4);
\coordinate (B) at (0.5,-5);
\coordinate (1) at (0,-1);
\coordinate (2) at (-2,0);
\coordinate (3) at (3,0);

\draw[fill=black] (A) circle (3.5pt);
\draw[fill=black] (B) circle (3.5pt);

\draw[fill=black] (1) circle (3.5pt);
\draw[fill=black] (2) circle (3.5pt);
\draw[fill=black] (3) circle (3.5pt);

\draw[thick, line width=1.75pt] (3)--(1)--(2);
\draw[thick, line width=1.75pt, dashed] (2)--(3);

\draw[thick] (A)--(1);
\draw[thick] (A)--(2);
\draw[thick] (A)--(3);

\draw[thick] (B)--(1);
\draw[thick] (B)--(2);
\draw[thick] (B)--(3);

\node at (-2.5,0) {$v_1$};
\node at (0.55,-1.4) {$v_2$};
\node at (3.58,0) {$v_3$};

\node at (0.5,4.5) {$a$};
\node at (0.5,-5.6) {$b$};

\end{tikzpicture}
\captionof{figure}{ }
\end{center}
Suppose that $n=2k+1$. If $k$ is an odd number, then $BP_n$ contains the unique zigzag (up to reversing)
$$av_1,v_1v_2,v_2b,bv_3,\dots,av_{n-2},v_{n-2}v_{n-1},v_{n-1}b,bv_n,v_nv_1,$$
$$v_1a,av_2,v_2v_3,v_3b,\dots, av_{n-1},v_{n-1}v_n,v_nb,$$
$$bv_1,v_1v_2,v_2a,av_3,\dots,bv_{n-2},v_{n-2}v_{n-1},v_{n-1}a,av_n,v_nv_1,$$
$$v_1b,bv_2,v_2v_3,v_3a,\dots,bv_{n-1},v_{n-1}v_n, v_na.$$
In the case when $k$ is even, then $BP_n$ also has the single zigzag (up to reversing)
$$av_1,v_1v_2,v_2b,bv_3,\dots,bv_{n-2},v_{n-2}v_{n-1}, v_{n-1}a,av_n,v_nv_1,$$
$$v_1b,bv_2,v_2v_3,v_3a,\dots,av_{n-1},v_{n-1}v_n,v_nb,$$
$$bv_1,v_1v_2,v_2a,av_3,\dots,av_{n-2},v_{n-2}v_{n-1},v_{n-1}b,bv_n,v_nv_1,$$
$$v_1a,av_2,v_2v_3,v_3b,\dots,bv_{n-1},v_{n-1}v_n, v_na.$$
Thus, for any odd $n$ the bipyramid $BP_n$ is $z$-knotted and has precisely one $z$-orientation (up to reversing). 
All edges of the base are of type II and the remaining edges are of type I.  
Each face is of type I and $BP_n$ is $z$-homogeneous. 
}\end{exmp}

\begin{exmp}\label{ex1b}{\rm
We assume that $n=2k$ and $k$ is odd. Zigzags of $BP_n$ are
$$av_1,v_1v_2,v_2b,bv_3,v_3v_4,\dots,av_{n-1},v_{n-1}v_n, v_nb,$$
$$bv_1,v_1v_2,v_2a,av_3,v_3v_4,\dots,bv_{n-1},v_{n-1}v_n, v_na$$
and 
$$av_2,v_2v_3,v_3b,bv_4,v_4v_5,\dots,av_n,v_nv_1,v_1b,$$
$$bv_2,v_2v_3,v_3a,av_4,v_4v_5,\dots,bv_n,v_nv_1,v_1a$$
and their reversions. 
We write $Z_1$ for the first and $Z_2$ for the second of the above zigzags. 
A direct verification shows that any edge of the base is of type II for any $z$-orientation. 
For the $z$-orientation $\tau=\{Z_1, Z_2\}$ the edges which does not belong to the base are of type I. 
Again, all faces are of type I and $(BP_n,\tau)$ is $z$-homogeneous. 
Observe that for the $z$-orientation $\tau'=\{Z_1, Z^{-1}_2\}$ the bipyramid $BP_n$ is not $z$-homogeneous. 
If $n=2k$ and $k$ is even, then $BP_n$ has the following four zigzags (up to reversing):
$$av_1,v_1v_2, v_2b,\dots,bv_{n-1},v_{n-1}v_n, v_na;$$ 
$$bv_1,v_1v_2, v_2a,\dots,av_{n-1},v_{n-1}v_n, v_nb;$$
$$av_2,v_2v_3,v_3b,\dots,bv_n,v_nv_1,v_1a;$$
$$bv_2,v_2v_3,v_3a,\dots,av_n,v_nv_1,v_1b.$$
For the $z$-orientation consisting of the above zigzags $BP_n$ is $z$-homogeneous and an edge is of type II only in the case when it belongs to the base. 
}\end{exmp}

\begin{exmp}\label{ex2}{\rm
We construct an infinite series of $z$-homogeneous spherical triangulations different from bipyramids. 
Let $\Gamma'$ be a graph embedded in $\mathbb{S}^2$ and consisting of two vertices $a,b$ and four paths $P_1,P_2,P_3,P_4$ such that at most one of these paths is an edge. 
We also require that any two of these paths intersect precisely in $a$ and $b$, see Fig.\!\! 3. 
Thus, $\mathbb{S}^2$ is decomposed into four disks whose boundaries are $P_1\cup P_2$, $P_2\cup P_3$, $P_3\cup P_4$ and $P_1\cup P_4$. 
We add a vertex in the interior of each of these disks and connect it by edges with the vertices of the boundary. 
We denote this spherical triangulation by $\Gamma_{p_1,p_2,p_3,p_4}$ where $p_i$ is the number of edges in the path $P_i$ (see Fig.\!\! 3 for $\Gamma_{2,3,3,3}$). 
Any cyclic permutation of $p_1,p_2,p_3,p_4$ gives an isomorphic triangulation. 
We orient the edges of all $P_i$ such that the boundary of every disk consists of a path from $a$ to $b$ and a path from $b$ to $a$. 
By the result mentioned above (\cite[Proposition 3]{T1}) $\Gamma_{p_1,p_2,p_3,p_4}$ is $z$-homogeneous for a certain $z$-orientation. 
\begin{center}
\begin{tikzpicture}[scale=0.6]
\coordinate[xshift=1cm] (A) at (30:5cm);
\coordinate (B) at (90:5cm);
\coordinate[xshift=-1cm] (C) at (150:5cm);
\coordinate[xshift=-1cm] (D) at (210:5cm);
\coordinate (E) at (270:5cm);
\coordinate[xshift=1cm] (F) at (330:5cm);

\draw[fill=black] (A) circle (4pt);
\draw[fill=black] (B) circle (4pt);
\draw[fill=black] (C) circle (4pt);
\draw[fill=black] (D) circle (4pt);
\draw[fill=black] (E) circle (4pt);
\draw[fill=black] (F) circle (4pt);

\coordinate[xshift=-0.5cm] (L) at ({-5*sqrt(3)/3.5},0);
\draw[fill=black] (L) circle (4pt);

\coordinate[xshift=0.3cm] (R1) at ({5*sqrt(3)/4},{5*sqrt(3)/3.5-1.2});
\coordinate[xshift=0.3cm] (R2) at ({5*sqrt(3)/4},{-5*sqrt(3)/3.5+1.2});
\draw[fill=black] (R1) circle (3pt);
\draw[fill=black] (R2) circle (3pt);

\coordinate (C1) at (0,{3.33/2});
\coordinate (C2) at (0,{-3.33/2});
\draw[fill=black] (C1) circle (4pt);
\draw[fill=black] (C2) circle (4pt);

\draw [thick, decoration={markings,
mark=at position 0.52 with {\arrow[scale=2,>=stealth]{>}},
mark=at position 0.58 with {\arrow[scale=2,>=stealth]{>}}},
postaction={decorate}] (B) -- (A);
\draw [thick, decoration={markings,
mark=at position 0.51 with {\arrow[scale=2,>=stealth]{>}},
mark=at position 0.59 with {\arrow[scale=2,>=stealth]{>}}},
postaction={decorate}] (A) -- (F);
\draw [thick, decoration={markings,
mark=at position 0.52 with {\arrow[scale=2,>=stealth]{>}},
mark=at position 0.58 with {\arrow[scale=2,>=stealth]{>}}},
postaction={decorate}] (F) -- (E);

\draw [thick, decoration={markings,
mark=at position 0.52 with {\arrow[scale=2,>=stealth]{>}},
mark=at position 0.58 with {\arrow[scale=2,>=stealth]{>}}},
postaction={decorate}] (D) -- (E);
\draw [thick, decoration={markings,
mark=at position 0.51 with {\arrow[scale=2,>=stealth]{>}},
mark=at position 0.59 with {\arrow[scale=2,>=stealth]{>}}},
postaction={decorate}] (C) -- (D);
\draw [thick, decoration={markings,
mark=at position 0.52 with {\arrow[scale=2,>=stealth]{>}},
mark=at position 0.58 with {\arrow[scale=2,>=stealth]{>}}},
postaction={decorate}] (B) -- (C);

\draw [thick, decoration={markings,
mark=at position 0.5 with {\arrow[scale=2,>=stealth]{>}},
mark=at position 0.6 with {\arrow[scale=2,>=stealth]{>}}},
postaction={decorate}] (B) -- (C1);
\draw [thick, decoration={markings,
mark=at position 0.5 with {\arrow[scale=2,>=stealth]{>}},
mark=at position 0.6 with {\arrow[scale=2,>=stealth]{>}}},
postaction={decorate}] (C1) -- (C2);
\draw [thick, decoration={markings,
mark=at position 0.5 with {\arrow[scale=2,>=stealth]{>}},
mark=at position 0.6 with {\arrow[scale=2,>=stealth]{>}}},
postaction={decorate}] (C2) -- (E);

\draw [thick, decoration={markings,
mark=at position 0.52 with {\arrow[scale=2,>=stealth]{>}},
mark=at position 0.58 with {\arrow[scale=2,>=stealth]{>}}},
postaction={decorate}] (L) -- (B);
\draw [thick, decoration={markings,
mark=at position 0.52 with {\arrow[scale=2,>=stealth]{>}},
mark=at position 0.58 with {\arrow[scale=2,>=stealth]{>}}},
postaction={decorate}] (E) -- (L);

\draw [thick, decoration={markings,
mark=at position 0.51 with {\arrow[scale=2,>=stealth]{>}},
mark=at position 0.59 with {\arrow[scale=2,>=stealth]{>}}},
postaction={decorate}] (R1) -- (B);
\draw [thick, decoration={markings,
mark=at position 0.49 with {\arrow[scale=2,>=stealth]{>}},
mark=at position 0.61 with {\arrow[scale=2,>=stealth]{>}}},
postaction={decorate}] (R2) -- (R1);
\draw [thick, decoration={markings,
mark=at position 0.51 with {\arrow[scale=2,>=stealth]{>}},
mark=at position 0.59 with {\arrow[scale=2,>=stealth]{>}}},
postaction={decorate}] (E) -- (R2);

\coordinate (RC1) at (-4.45,0);
\draw[fill=black] (RC1) circle (3pt);

\draw [dashed] (RC1) -- (L);
\draw [dashed] (RC1) -- (B);
\draw [dashed] (RC1) -- (C);
\draw [dashed] (RC1) -- (D);
\draw [dashed] (RC1) -- (E);

\coordinate (RC2) at (-1.5,0);
\draw[fill=black] (RC2) circle (3pt);

\draw [dashed] (RC2) -- (B);
\draw [dashed] (RC2) -- (C1);
\draw [dashed] (RC2) -- (C2);
\draw [dashed] (RC2) -- (E);
\draw [dashed] (RC2) -- (L);

\coordinate (RC3) at (1.4,0);
\draw[fill=black] (RC3) circle (3pt);

\draw [dashed] (RC3) -- (B);
\draw [dashed] (RC3) -- (C1);
\draw [dashed] (RC3) -- (C2);
\draw [dashed] (RC3) -- (E);
\draw [dashed] (RC3) -- (R2);
\draw [dashed] (RC3) -- (R1);

\coordinate (RC4) at (4.6,0);
\draw[fill=black] (RC4) circle (3pt);

\draw [dashed] (RC4) -- (B);
\draw [dashed] (RC4) -- (E);
\draw [dashed] (RC4) -- (R2);
\draw [dashed] (RC4) -- (R1);
\draw [dashed] (RC4) -- (A);
\draw [dashed] (RC4) -- (F);

\draw[thick, line width=2pt] (B) -- (A) -- (F) -- (E) -- (R2) -- (R1) -- (B) -- (C1) -- (C2) -- (E) -- (L) -- (B) -- (C) -- (D) -- (E);

\node at (3.2,4.2) {$e_1$};
\node at (6.6,0) {$e_2$};
\node at (3.2,-4.2) {$e_3$};

\node at (-3.2,4.2) {$e_1$};
\node at (-6.6,0) {$e_2$};
\node at (-3.2,-4.2) {$e_3$};

\node at (0,5.4) {$a$};
\node at (0,-5.45) {$b$};

\node at (-6.5,1.9) {$P_4$};
\node at (6.5,1.9) {$P_4$};
\node at (-1.9,1.2) {$P_1$};
\node at (-0.55,0) {$P_2$};
\node at (3.25,0) {$P_3$};
\end{tikzpicture}
\captionof{figure}{ }
\end{center}
}\end{exmp}

\section{Main result}

We describe the connected sum of special type of two triangulated surfaces. 
Let $(\Gamma,\tau)$ be a $z$-homogeneous triangulation of the surface $M$ for a certain $z$-orientation $\tau$. 
In the graph $\Gamma_{II}$ we consider a path $P$ consisting of two edges $e_1=v_1v_2$ and $e_2=v_2v_3$, see Fig. 4. 
In what follows, any such pair will be called {\it special}. 
Denote by $F^{+}_{i}$ and $F^{-}_{i}$ the faces containing the edge $e_i$ for $i=1,2$. 
We assume that for every $\delta\in \{+,-\}$ the faces $F^{\delta}_{1}$ and $F^{\delta}_{2}$ are on the same side of the path $P$. 
We split up each $e_{i}$ in two edges $e^{+}_{i}$ and $e^{-}_{i}$ (whose direction is coincident with the direction of $e_i$) such that the face $F^{\delta}_{i}$ contains the edge $e^{\delta}_{i}$. 
The vertex $v_2$ is splitted in two vertices $v_2^{+}$ and $v_2^{-}$ such that $e^{\delta}_{1}=v_1v_2^{\delta}$ and $e^{\delta}_{2}=v_2^{\delta}v_3$ for $\delta\in \{+,-\}$ (see Fig. 4). 
We obtain a new graph $N_P(\Gamma)$ embedded in $M$ whose faces are all the faces of $\Gamma$ and a new $4$-gonal face $F_P$. 
\begin{center}
\begin{tikzpicture}[scale=0.6]

\begin{scope}
\coordinate (A) at (0,0);
\coordinate (B) at (5cm,0);
\coordinate (C) at (10cm,0);

\draw[fill=black] (A) circle (4pt);
\draw[fill=black] (B) circle (4pt);
\draw[fill=black] (C) circle (4pt);

\draw[thick, line width=2pt] (A) -- (B);
\draw [thick, decoration={markings,
mark=at position 0.52 with {\arrow[scale=2,>=stealth]{>}},
mark=at position 0.58 with {\arrow[scale=2,>=stealth]{>}}},
postaction={decorate}] (A) -- (B);

\draw[thick, line width=2pt] (B) -- (C);
\draw [thick, decoration={markings,
mark=at position 0.52 with {\arrow[scale=2,>=stealth]{>}},
mark=at position 0.58 with {\arrow[scale=2,>=stealth]{>}}},
postaction={decorate}] (B) -- (C);

\node at (2.5cm,1.4cm) {$F^+_1$};
\node at (7.5cm,1.4cm) {$F^+_2$};
\node at (2.5cm,-1.4cm) {$F^-_1$};
\node at (7.5cm,-1.4cm) {$F^-_2$};

\node at (2.5cm,-0.6cm) {$e_1$};
\node at (7.5cm,-0.6cm) {$e_2$};

\node at (0,-0.6cm) {$v_1$};
\node at (5cm,-0.6cm) {$v_2$};
\node at (10cm,-0.6cm) {$v_3$};
\end{scope}

\begin{scope}[yshift=-8cm]
\coordinate (A) at (0,0);
\coordinate (B+) at (5cm,1.75cm);
\coordinate (B-) at (5cm,-1.75cm);
\coordinate (C) at (10cm,0);

\draw[fill=black] (A) circle (4pt);
\draw[fill=black] (B+) circle (4pt);
\draw[fill=black] (B-) circle (4pt);
\draw[fill=black] (C) circle (4pt);

\draw[thick, line width=2pt] (A) -- (B+);
\draw [thick, decoration={markings,
mark=at position 0.52 with {\arrow[scale=2,>=stealth]{>}},
mark=at position 0.58 with {\arrow[scale=2,>=stealth]{>}}},
postaction={decorate}] (A) -- (B+);

\draw[thick, line width=2pt] (A) -- (B-);
\draw [thick, decoration={markings,
mark=at position 0.52 with {\arrow[scale=2,>=stealth]{>}},
mark=at position 0.58 with {\arrow[scale=2,>=stealth]{>}}},
postaction={decorate}] (A) -- (B-);

\draw[thick, line width=2pt] (B+) -- (C);
\draw [thick, decoration={markings,
mark=at position 0.52 with {\arrow[scale=2,>=stealth]{>}},
mark=at position 0.58 with {\arrow[scale=2,>=stealth]{>}}},
postaction={decorate}] (B+) -- (C);

\draw[thick, line width=2pt] (B-) -- (C);
\draw [thick, decoration={markings,
mark=at position 0.52 with {\arrow[scale=2,>=stealth]{>}},
mark=at position 0.58 with {\arrow[scale=2,>=stealth]{>}}},
postaction={decorate}] (B-) -- (C);

\node at (2.3cm,2.5cm) {$F^+_1$};
\node at (7.7cm,2.5cm) {$F^+_2$};
\node at (2.3cm,-2.5cm) {$F^-_1$};
\node at (7.7cm,-2.5cm) {$F^-_2$};

\node at (2.1cm,1.4cm) {$e^+_1$};
\node at (7.9cm,1.4cm) {$e^+_2$};
\node at (2.1cm,-1.4cm) {$e^-_1$};
\node at (7.9cm,-1.4cm) {$e^-_2$};

\node at (0,-0.6cm) {$v_1$};
\node at (5cm,2.3cm) {$v^+_2$};
\node at (5cm,-2.3cm) {$v^-_2$};
\node at (10cm,-0.6cm) {$v_3$};

\node at (5cm,0) {$F_P$};
\end{scope}

\draw[thick, line width=0.75pt] (5cm,-1.5cm) -- (5cm,-4.8cm);
\draw [thick, decoration={markings,
mark=at position 1 with {\arrow[scale=2,>=stealth]{>}}},
postaction={decorate}] (5cm,-1.5cm) -- (5cm,-5cm);

\end{tikzpicture}
\captionof{figure}{ }
\end{center}
Let $(\Gamma', \tau')$ be another $z$-homogeneous triangulation of a surface $M'$ and let $\Gamma'_{II}$ be the subgraph of $\Gamma'$ formed by all edges of type II and their vertices. 
We choose a special pair $P'$ in $\Gamma'_{II}$ and, as above, construct the graph $N_{P'}(\Gamma')$. 
We need the following additional assumption: 
\begin{enumerate}
\item[$(*)$] for at least one of the pairs $P,P'$ there is no edge in $\Gamma$ or $\Gamma'$ (respectively) which connects the endpoints. 
\end{enumerate}
In other words, at least one of the pairs is not a part in a cycle of length $3$. 

It clear that the $z$-orientations $\tau$ and $\tau'$ determine the orientation on edges of $F_P$ and $F_{P'}$ (respectively). 
We denote by $\omega(F_P)$ and $\omega(F_{P'})$ the sets of such oriented edges for $F_P$ and $F_{P'}$, respectively. 
We say that a homeomorphism $g:\partial F_P\to\partial F_{P'}$ is {\it special} if it transfers the vertices of $F_P$ to the vertices of $F_{P'}$ and the edges from $\omega(F_P)$ to the edges from $\omega(F_{P'})$. 

Now, we remove the interiors of $F_P$ and $F_{P'}$ from $M$ and $M'$ (respectively). 
Using any special homeomorphism $g:\partial F_P\to\partial F_{P'}$, we glue the boundaries $\partial F_P$ and $\partial F_{P'}$. 
We obtain a triangulation of $M\#M'$ and denote it by $\Gamma\#_{g}\Gamma'$. 
It must be pointed out that the assumption $(*)$ guarantees that there are no multiple edges in $\Gamma\#_{g}\Gamma'$. 
Since the special homeomorphism $g$ preserves the orientation, the connected sum $\Gamma\#_{g}\Gamma'$ admits a $z$-orientation such that each edge in $\Gamma\#_{g}\Gamma'$ has the same type as it has in $(\Gamma, \tau)$ or $(\Gamma', \tau')$. 
For this $z$-orientation our connected sum is $z$-homogeneous. 

The above construction can be applied several times. 
Let $P_1,\dots, P_n$ be special pairs in $(\Gamma,\tau)$ and let $P'_1, \dots, P'_n$ be special pairs in spherical $z$-homogeneous triangulations $(\Gamma_1,\tau_1),\dots,(\Gamma_n,\tau_n)$. 
For any special homeomorphism $g_i:\partial F_{P_i}\to\partial F_{P'_i}$ with $i\in\{1,\dots,n\}$ the connected sum
\begin{equation}\label{eq-1}
(((\Gamma \#_{g_{1}} \Gamma_{1})\#_{g_{2}} \Gamma_{2})\dots)\#_{g_{n}}\Gamma_{n}
\end{equation}
is a $z$-homogeneous triangulation of $M$ (for a certain $z$-orientation). 

\begin{theorem}\label{th3}
For any $z$-homogeneous triangulation $(\Gamma,\tau)$ there exist spherical $z$-homogeneous triangulations $(\Gamma_1,\tau_1),\dots,(\Gamma_n,\tau_n)$ such that the $z$-homogeneous connected sum \eqref{eq-1} is $z$-knotted for some special pairs $P_i$ in $\Gamma$, $P'_i$ in $\Gamma_i$ and special homeomorphisms $g_i:\partial F_{P_i}\to\partial F_{P'_i}$ with $i\in\{1,\dots,n\}$.
\end{theorem}

\section{$Z$-monodromy}

As in the previous sections, we consider a closed $2$-cell embedding of a connected simple finite graph without vertices of degree $2$. 
Let $F$ be a $k$-gonal face of this graph. 
Suppose that $x_1,\dots,x_k$ are the consecutive vertices of $F$ and consider the set of all oriented edges
$$\Omega(F)=\{x_1x_2,\dots,x_kx_1,x_kx_1,\dots,x_2x_1\}.$$
If $e=xy$, then we write $-e$ for the edge directed from $y$ to $x$.
Let $D_F$ be the permutation on $\Omega(F)$ which transfers $xx'$ to $x'x''$ for any three consecutive vertices $x,x',x''$ belonging to $F$.
So, $D_F$ is the composition of two distinct commuting $k$-cycles and any of these cycles can be identified with one of the orientations on the boundary of $F$. 

For any $e\in \Omega(F)$ we take $e_{0}\in \Omega(F)$ such that $D_{F}(e_{0})=e$ and the zigzag containing the sequence $e_{0},e$. 
We define $M_{F}(e)$ as the first element of $\Omega(F)$ contained in this zigzag after $e$. 
The permutation $M_F$ of the set $\Omega(F)$ is called a {\it $z$-monodromy} of the face $F$. 
A full description of all $z$-monodromies for triangle faces is given in \cite{PT2}. 
Now, using the concept of $z$-monodromy for faces, we define the $z$-monodromy for a pair of consecutive edges of type II (in a $z$-homogeneous triangulation) oriented in the same direction. 

As before, suppose that $(\Gamma,\tau)$ is a $z$-homogeneous triangulation and $P$ is a special pair in $\Gamma_{II}$. 
For a $4$-gonal face $F_P$ in $N_P(\Gamma)$ we consider the set of all oriented edges
$$\Omega(F_P)=\{e^{+}_{1},e^{+}_{2},e^{-}_{1},e^{-}_{2},-e^{+}_{1},-e^{+}_{2},-e^{-}_{1},-e^{-}_{2}\}$$ 
and the associated $z$-monodromy $M_{F_P}$. 
The restriction of $M_{F_P}$ to the set 
$$\omega{(F_P)}=\{e^{+}_{1},e^{+}_{2},e^{-}_{1},e^{-}_{2}\}$$ 
will be called the {\it $z$-monodromy} of $P$ and denoted by $M_P$. 

We show that $M_P$ is a permutation on $\omega{(F_P)}$. 
For every $e^{\delta}_{i}\in\omega{(F_P)}$ we consider the following part of a zigzag in $N_P(\Gamma)$
$$[e^{\delta}_{i}, D_{F^{\delta}_{i}}(e^{\delta}_{i}),\dots,M_P(e^{\delta}_{i})].$$ 
It corresponds to a part of a zigzag from $\tau$ which is
$$[e_i, D_{F^{\delta}_{i}}(e_{i}),\dots,e_j],$$ 
where $j\in\{1,2\}$. 
This implies that $M_P(e^{\delta}_{i})=e^{\gamma}_j$ for $\gamma\in\{+,-\}$ and, consequently, $M_P(e^{\delta}_{i})\in\omega(F_P)$. 
So, the $z$-monodromy $M_P$ can be identified with a certain permutation from the symmetric group $S_4$. 

From this moment, for simplicity, we write $1, 2$ for $e^{+}_{1},e^{+}_{2}$ and $3,4$ for $e^{-}_{1},e^{-}_{2}$ respectively. 
Since $|S_4|=24$, there are $24$ possibilities for $M_P$ (at this moment we do not discuss their realization). 
We decompose these possibilities on the following $13$ classes: 
\begin{enumerate}
\item[$(K_0)$] $id$,

\item[$(K_1)$] $(1234)$,
\item[$(K_2)$] $(13)(24)$,
\item[$(K_3)$] $(1432)$,
\item[$(K_4)$] $(14)(23)$,

\item[$(K_5)$] $(12)(34)$,
\item[$(K_6)$] $(24), (13)$,
\item[$(K_7)$] $(34), (12)$,
\item[$(K_8)$] $(23), (14)$,

\item[$(K_9)$] $(1324), (1423)$,
\item[$(K_{10})$] $(1243), (1342)$,
\item[$(K_{11})$] $(234), (123), (124), (134)$, 
\item[$(K_{12})$] $(243), (132), (142), (143)$.
\end{enumerate}
Now, we explain why the elements of $S_4$ are separated in this way. 
We will use the following two commuting involutions:
$$s=(13)(24),\hspace{1cm}t=(12)(34).$$ 
Two permutations $p_1, p_2$ belong to the same class $K_i$, $i\geq6$, if one of the following possibilities is realized: 
$$p_2=sp_1s\hspace{0.5cm}\text{or}\hspace{0.5cm}p_2=up^{-1}_1u,$$ 
where $u\in\{t,st\}$ (see Tab. 1 in Appendix for a verification). 
The permutation $s$ corresponds to the transposing of $+$ and $-$ for the edges $e^{\delta}_i$. 
If we replace the $z$-orientation $\tau$ by $\tau^{-1}$, then a permutation $p$ corresponding to the $z$-monodromy is changed by $tp^{-1}t$. 

The following examples show that each of the mentioned above $13$ classes of $z$-monodromy is realized. 
\begin{exmp}\label{ex3}{\rm
Suppose that $\Gamma$ is the bypiramid $BP_n$ (with the $z$-orientation from Examples \ref{ex1} and \ref{ex1b}). 
Let $P$ be a special pair formed by the directed edges $v_1v_2$ and $v_2v_3$. 
We denote these edges by $1$ and $2$, respctively, if they are considered as the edges of the faces containing $a$ in $N_P(BP_n)$. 
If these edges are considered as the edges on the faces containing $b$ in $N_P(BP_n)$, then we will write for them $3$ and $4$ respectively.

{\it The case when $n=2k+1$ and $k$ is odd.} The zigzag passes through $P$ as follows
$$\dots,av_1,v_1v_2,v_2b,\dots,av_2,v_2v_3,v_3b,\dots,bv_1,v_1v_2,v_2a,\dots,bv_2,v_2v_3,v_3a,\dots$$
This zigzag defines the following union of parts of zigzags in $N_P(BP_n)$: 
$$\dots,1,3,\dots,2,4,\dots,3,1,\dots,4,2,\dots$$
and the $z$-monodromy is 
$$M_P=(1432)$$
from the class $K_3$.

{\it The case when $n=2k+1$ and $k$ is even.} The zigzag in $BP_n$ is 
$$\dots,av_1,v_1v_2,v_2b,\dots,bv_2,v_2v_3,v_3a,\dots,bv_1,v_1v_2,v_2a,\dots,av_2,v_2v_3,v_3b,\dots$$
and the corresponding union of parts of zigzags in $N_P(BP_n)$ is
$$\dots,1,3,\dots,4,2,\dots,3,1,\dots,2,4,\dots$$
which means that the $z$-monodromy is
$$M_P=(1234)$$
from the class $K_1$.

{\it The case when $n=2k$ and $k$ is odd.} The edges of $P$ occur in zigzags of $BP_n$ as follows: 
$$\dots,av_1,v_1v_2,v_2b,\dots,bv_1,v_1v_2,v_2a,\dots$$
and
$$\dots,av_2,v_2v_3,v_3b,\dots,bv_2,v_2v_3,v_3a,\dots;$$
the corresponding unions of parts of zigzags in $N_P(BP_n)$ are
$$\dots,1,3,\dots,3,1,\dots\hspace{0.5cm}\text{ and }\hspace{0.5cm}\dots,2,4,\dots,4,2,\dots$$ 
which implies that $M_P$ is the identity (class $K_0$). 

{\it The case when $n=2k$ and $k$ is even.} The four zigzags of $BP_n$ pass through the edges of $P$ as follows 
$$\dots,av_1,v_1v_2,v_2b,\dots,$$ 
$$\dots,bv_1,v_1v_2,v_2a,\dots,$$ 
$$\dots,av_2,v_2v_3,v_3b,\dots,$$ 
$$\dots,bv_2,v_2v_3,v_3a,\dots$$ 
and the corresponding parts of zigzags in $N_P(BP_n)$ are 
$$\dots,1,3,\dots,\hspace{0.6cm}\dots,3,1,\dots,\hspace{0.6cm}\dots,2,4,\dots,\hspace{0.6cm}\dots,4,2,\dots$$ 
and the $z$-monodromy 
$$M_P=(13)(24)$$ 
is from the class $K_2$. 

Any special pair of $BP_n$ can be transferred to any another special pair by an automorphism and for all special pairs the $z$-monodromies belong to the same class.  
}\end{exmp}

\begin{exmp}\label{ex4}{\rm 
Consder the $z$-homogeneous triangulation $\Gamma_{2,3,4,5}$ (with the corresponding $z$-orientation) as in Example \ref{ex2} with the paths 
$$P_1=\{a,v_0,b\},$$ 
$$P_2=\{b,v_1,v_2,a\},$$ 
$$P_3=\{a,v_3,v_4,v_5,b\},$$ 
$$P_4=\{b,v_6,v_7,v_8,v_9,a\}.$$ 
Denote by $v_{ij}$ the vertex in the interior whose boundary is $P_i\cup P_j$, $i,j\in\{1,2,3,4\}$ and $i<j$. 
The $z$-orientation of $\Gamma_{2,3,4,5}$ consists of the follownig two zigzags (we present them as cyclic sequences of vertices):
$$v_{14},a,v_0,v_{12},b,v_1,v_{23},v_2,a,v_{12},v_0,b,v_{14},v_6,v_7,v_{34},v_8,v_9$$
and
$$v_{14},v_0,b,v_{12},v_1,v_2,v_{23},a,v_3,v_{34},v_4,v_5,v_{23},b,v_1,v_{12},v_2,a,v_{23},v_3,v_4,$$
$$v_{34},v_5,b,v_{23},v_1,v_2,v_{12},a,v_0,v_{14},b,v_6,v_{34},v_7,v_8,v_{14},v_9,a,v_{34},v_3,v_4,$$
$$v_{23},v_5,b,v_{34},v_6,v_7,v_{14},v_8,v_9,v_{34},a,v_3,v_{23},v_4,v_5,v_{34},b,v_6,v_{14},v_7,v_8,v_{34},v_9,a.$$ 

$(1).$ If $P$ is the special pair $v_0b,bv_1$, then the zigzags 
$$\dots,v_{12}b,bv_1,v_1v_{23},\dots,v_{12}v_0,v_0b,bv_{14},\dots$$
and
$$\dots,v_{14}v_0,v_0b,bv_{12},\dots,v_{23}b,bv_1,v_1v_{12},\dots$$
pass through the edges of $P$. 
In $N_P(\Gamma_{2,3,4,5})$, we denote by $1,2$ the edges corresponding to $v_0b,bv_1$ (respectively) which are in the faces containing $v_{12}$. 
We write $3,4$ for the edges corresponding to $v_0b,bv_1$ which are in the faces containing $v_{14}$ and $v_{23}$, respectively. 
The associated unions of parts of zigzags in $N_P(\Gamma_{2,3,4,5})$ are 
$$\dots,2,4,\dots,1,3,\dots\hspace{0.5cm}\text{and}\hspace{0.5cm}\dots,3,1,\dots,4,2,\dots$$ 
which implies that the $z$-monodromy
$$M_P=(14)(23)$$
is from the class $K_{4}$.

$(2).$ If $P$ is the special pair $av_0,v_0b$, then the zigzags 
$$\dots,v_{14}a,av_0,v_0v_{12},\dots,v_{12}v_0,v_0b,bv_{14},\dots$$
and
$$\dots,v_{14}v_0,v_0b,bv_{12},\dots,v_{12}a,av_0,v_0v_{14},\dots$$
pass through the edges of $P$. 
In $N_P(\Gamma_{2,3,4,5})$, we denote by $1,2$ the edges corresponding to $av_0,v_0b$ (respectively) which are in the faces containing $v_{12}$. 
We write $3,4$ for the edges related to $av_0,v_0b$ which are in the faces containing $v_{14}$, respectively. 
The corresponding unions of parts of zigzags in $N_P(\Gamma_{2,3,4,5})$ are 
$$\dots,3,1,\dots,2,4,\dots\hspace{0.5cm}\text{and}\hspace{0.5cm}\dots,4,2,\dots,1,3,\dots$$ 
which implies that the $z$-monodromy
$$M_P=(12)(34)$$
is from the class $K_{5}$.

$(3).$ If $P$ is the special pair $v_3v_4,v_4v_5$, then the zigzag 
$$\dots,v_{34}v_4,v_4v_5,v_5v_{23},\dots,v_{23}v_3,v_3v_4,v_4v_{34},\dots,v_{34}v_3,v_3v_4,v_4v_{23},\dots,v_{23}v_4,v_4v_5,v_5v_{34},\dots$$
passes through the edges of $P$. 
In $N_P(\Gamma_{2,3,4,5})$, we denote by $1,2$ the edges corresponding to $v_3v_4,v_4v_5$ (respectively) which are in the faces containing $v_{23}$. 
We write $3,4$ for the edges associated to $v_3v_4,v_4v_5$ which are in the faces containing $v_{34}$, respectively. 
The related union of parts of zigzags in $N_P(\Gamma_{2,3,4,5})$ is 
$$\dots, 4,2, \dots, 1,3, \dots, 3,1, \dots, 2,4,\dots$$
and the $z$-monodromy
$$M_P=(12)$$
is from the class $K_{7}$.

$(4).$ If $P$ is the special pair $av_3,v_3v_4$, then the zigzag 
$$\dots,v_{23}a,av_3,v_3v_{34},\dots,v_{23}v_3,v_3v_4,v_4v_{34},\dots,v_{34}v_3,v_3v_4,v_4v_{23},\dots,v_{34}a,av_3,v_3v_{23},\dots$$ 
passes through the edges of $P$. 
In $N_P(\Gamma_{2,3,4,5})$, we denote by $1,2$ the edges corresponding to $av_3,v_3v_4$ (respectively) which are in the faces containing $v_{23}$. 
We write $3,4$ for the edges corresponding to $av_3,v_3v_4$ which are in the faces containing $v_{34}$, respectively. 
The associated union of parts of zigzags in $N_P(\Gamma_{2,3,4,5})$ is 
$$\dots, 1,3,\dots,2,4,\dots,4,2,\dots, 3,1,\dots$$
and the $z$-monodromy
$$M_P=(23)$$
is from the class $K_{8}$.

$(5).$ If $P$ is the special pair $bv_6,v_6v_7$, then the zigzags 
$$\dots,v_{14}v_6,v_6v_7,v_7v_{34}\dots$$
and
$$\dots,v_{14}b,bv_6,v_6v_{34},\dots,v_{34}v_6,v_6v_7,v_7v_{14},\dots,v_{34}b,bv_6,v_6v_{14},\dots$$ 
pass through the edges of $P$. 
In $N_P(\Gamma_{2,3,4,5})$, we denote by $1,2$ the edges corresponding to $bv_6,v_6v_7$ (respectively) which are in the faces containing $v_{14}$. 
We write $3,4$ for the edges related to $bv_6,v_6v_7$ which are in the faces containing $v_{34}$, respectively. 
The corresponding unions of parts of zigzags in $N_P(\Gamma_{2,3,4,5})$ are 
$$\dots,2,4,\dots\hspace{0.5cm}\text{and}\hspace{0.5cm}\dots,1,3,\dots,4,2,\dots,3,1,\dots$$ 
which implies that the $z$-monodromy
$$M_P=(234)$$
is from the class $K_{11}$.

$(6).$ If $P$ is the special pair $bv_1,v_1v_2$, then the zigzags 
$$\dots,v_{12}b,bv_1,v_1v_{23},\dots$$
and
$$\dots,v_{12}v_1,v_1v_2,v_2v_{23},\dots,v_{23}b,bv_1,v_1v_{12},\dots,v_{23}v_1,v_1v_2,v_2v_{12},\dots$$
pass through the edges of $P$. 
In $N_P(\Gamma_{2,3,4,5})$, we denote by $1,2$ the edges corresponding to $bv_1,v_1v_2$ (respectively) which are in the faces containing $v_{12}$. 
We write $3,4$ for the edges corresponding to $bv_1,v_1v_2$ which are in the faces containing $v_{23}$, respectively. 
The associated unions of parts of zigzags in $N_P(\Gamma_{2,3,4,5})$ are 
$$\dots,1,3,\dots\hspace{0.5cm}\text{and}\hspace{0.5cm}\dots,2,4,\dots,3,1,\dots,4,2,\dots$$ 
which implies that the $z$-monodromy 
$$M_P=(143)$$
is from the class $K_{12}$.
}\end{exmp}

\begin{exmp}\label{ex5}{\rm
Now, we consder the $z$-homogeneous triangulation $\Gamma_{2,4,3,4}$ (with the corresponding $z$-orientation) from Example \ref{ex2} with the paths 
$$P_1=\{a,v_0,b\},$$
$$P_2=\{b,v_1,v_2,v_3,a\},$$
$$P_3=\{a,v_4,v_5,b\},$$
$$P_4=\{b,v_6,v_7,v_8,a\}.$$
As in the previous example, $v_{ij}$ is the vertex in the interior whose boundary is $P_i\cup P_j$, $i,j\in\{1,2,3,4\}$ and $i<j$. 
The $z$-orientation of $\Gamma_{2,4,3,4}$ consists of the follownig three zigzags:
$$v_{14},a,v_0,v_{12},b,v_1,v_{23},v_2,v_3,v_{12},a,v_0,v_{14},b,v_6,v_{34},v_7,v_8$$
and
$$v_{14},v_0,b,v_{12},v_1,v_2,v_{23},v_3,a,v_{12},v_0,b,v_{14},v_6,v_7,v_{34},v_8,a$$
and
$$v_{23},b,v_1,v_{12},v_2,v_3,v_{23},a,v_4,v_{34},v_5,b,v_{23},v_1,v_2,v_{12},v_3,a,v_{23},v_4,v_5,$$
$$v_{34},b,v_6,v_{14},v_7,v_8,v_{34},a,v_4,v_{23},v_5,b,v_{34},v_6,v_7,v_{14},v_8,a,v_{34},v_4,v_5.$$

$(1).$ If $P$ is the special pair $v_0b,bv_1$, then the zigzags 
$$\dots,v_{12}b,bv_1,v_1v_{23},\dots$$
and
$$\dots,v_{14}v_0,v_0b,bv_{12},\dots,v_{12}v_0,v_0b,bv_{14},\dots$$
and
$$\dots,v_{23}b,bv_1,v_1v_{12},\dots$$
pass through the edges of $P$. 
In $N_P(\Gamma_{2,4,3,4})$, we denote by $1,2$ the edges corresponding to $v_0b,bv_1$ (respectively) which are in the faces containing $v_{12}$. 
We write $3,4$ for the edges corresponding to $v_0b,bv_1$ which are in the faces containing $v_{14}$ and $v_{23}$, respectively. 
The associated unions of parts of zigzags in $N_P(\Gamma_{2,4,3,4})$ are 
$$\dots,2,4,\dots,\hspace{0.6cm}\dots,3,1,\dots,1,3,\dots,\hspace{0.6cm}\dots,4,2,\dots$$ 
and the $z$-monodromy
$$M_P=(24)$$
is from the class $K_{6}$.

$(2).$ If $P$ is the special pair $bv_1,v_1v_2$ then the zigzags 
$$\dots,v_{12}b,bv_1,v_1v_{23},\dots$$ 
and 
$$\dots,v_{12}v_1,v_1v_2,v_2v_{23},\dots$$
and
$$\dots,v_{23}b,bv_1,v_1v_{12},\dots,v_{23}v_1,v_1v_2,v_2v_{12},\dots$$
pass through the edges of $P$. 
In $N_P(\Gamma_{2,4,3,4})$, we denote by $1,2$ the edges corresponding to $bv_1,v_1v_2$ (respectively) which are in the faces containing $v_{12}$. 
We write $3,4$ for the edges related to $bv_1,v_1v_2$ which are in the faces containing $v_{23}$, respectively. 
The corresponding unions of parts of zigzags in $N_P(\Gamma_{2,4,3,4})$ are 
$$\dots,1,3,\dots,\hspace{0.6cm}\dots,2,4,\dots,\hspace{0.6cm}\dots,3,1,\dots,4,2,\dots$$ 
which implies that the $z$-monodromy
$$M_P=(1423)$$
is from the class $K_{9}$.

$(3).$ If $P$ is the special pair $v_1v_2,v_2v_3$ then the zigzags 
$$\dots,v_{23}v_2,v_2v_3,v_3v_{12},\dots$$
and
$$\dots,v_{12}v_1,v_1v_2,v_2v_{23},\dots$$
and
$$\dots,v_{12}v_2,v_2v_3,v_3v_{23},\dots,v_{23}v_1,v_1v_2,v_2v_{12},\dots$$
pass through the edges of $P$. 
In $N_P(\Gamma_{2,4,3,4})$, we denote by $1,2$ the edges corresponding to $v_1v_2,v_2v_3$ (respectively) which are in the faces containing $v_{12}$. 
We write $3,4$ for the edges associated to $v_1v_2,v_2v_3$ which are in the faces containing $v_{23}$, respectively. 
The related unions of parts of zigzags in $N_P(\Gamma_{2,4,3,4})$ are 
$$\dots,4,2,\dots,\hspace{0.6cm}\dots,1,3,\dots,\hspace{0.6cm}\dots,2,4,\dots,3,1,\dots$$ 
and the $z$-monodromy
$$M_P=(1243)$$ 
is from the class $K_{10}$.
}\end{exmp}

\section{Proof of Theorem 1}
Let $\Gamma$ and $\Gamma'$ be $z$-homogeneous triangulations (with fixed $z$-orientation). 
Let also $P$ and $P'$ be special pairs in $\Gamma$ and $\Gamma'$, respectively. 
As above, we write $F_P, F_{P'}$ for the faces in $N_P(\Gamma), N_{P'}(\Gamma')$ obtained from $P$ and $P'$, respectively.

\begin{lemma}\label{lem2}
The number of zigzags from the $z$-orientation of $\Gamma$ passing through the edges of $P$ is equal to the number of cycles in the permutation $sM_P$ where $s=(13)(24)$. 
\end{lemma}
\begin{proof}
Let $x\in\omega(F_P)=\{1,2,3,4\}$. 
Denote by $[x,M_P(x)]$ the part of a zigzag in $N_P{(\Gamma})$ whose the first element is $x$, the last element is $M_P(x)$ and this path does not contain any other elements from $\omega(F_P)$. 
Consider the cyclic sequence 
\begin{equation}\label{eq-2}
[x,M_P(x)], [sM_P(x),M_PsM_P(x)], \dots, [(sM_P)^{m-1}(x),M_P(sM_P)^{m-1}(x)],
\end{equation}
where $m$ is the smallest positive number such that $(sM_P)^{m}(x)=x$. 
Observe that the edges $y,sy\in\omega(F_P)$ correspond to the same edge of $\Gamma$. 
Therefore, if we identify the last edge of every part of \eqref{eq-2} with the first edge of the next part, then we obtain a zigzag in $\Gamma$. 
So, every cycle in $sM_P$ defines a zigzag passing through the edges of $P$ (it is easy to see that other cycle of $sM_P$ gives a different zigzag). 
Conversely, any zigzag of $\Gamma$ passing through the edges of $P$ induces a cycle in $sM_P$ by the definition of $z$-monodromy. 
\end{proof}

\begin{lemma}\label{lem3}
If $k$ is the number of zigzags from the $z$-orientation passing through the edges of $P$, then the following assertions are fulfilled:
\begin{enumerate}
\item[$\bullet$] $k=1$ if the $z$-monodromy of $P$ is from the class $K_i$ with $i\in\{1,3,7,8\}$; 
\item[$\bullet$] $k=2$ if the $z$-monodromy of $P$ is from the class $K_i$ with $i\in\{0,4,5,11,12\}$; 
\item[$\bullet$] $k=3$ if the $z$-monodromy of $P$ is from the class $K_i$ with $i\in\{6,9,10\}$; 
\item[$\bullet$] $k=4$ if the $z$-monodromy of $P$ is from the class $K_2$.
\end{enumerate}
\end{lemma}

\begin{proof}
If $K_i$ contains more than one permutation, then any two permutations $p,q\in K_i$ have the same number of cycles (since $q$ is conjugate to $p$ or $p^{-1}$). 
So, we choose an element from every of the $13$ classes and apply Lemma \ref{lem2}. 
We get the following: 
\begin{enumerate}
\item[$(K_0)$] $s(id)=s=(13)(24)$,
\item[$(K_1)$] $s(1234)=(1432)$,
\item[$(K_2)$] $s(13)(24)=s^2=id$,
\item[$(K_3)$] $s(1432)=(1234)$,
\item[$(K_4)$] $s(14)(23)=(12)(34)$,
\item[$(K_5)$] $s(12)(34)=(14)(23)$,
\item[$(K_6)$] $s(24)=(13)$,
\item[$(K_7)$] $s(34)=(1324)$,
\item[$(K_8)$] $s(23)=(1342)$,
\item[$(K_9)$] $s(1324)=(34)$,
\item[$(K_{10})$] $s(1243)=(14)$,
\item[$(K_{11})$] $s(234)=(132)$,
\item[$(K_{12})$] $s(243)=(134)$.
\end{enumerate}
Note that each fixed point is a $1$-cycle. 
\end{proof}

A special pair $P$ is said to be {\it essential} if every zigzag from the $z$-orientation passes through the edges of this pair.  

\begin{exmp}\label{ex6}{\rm
Consider the following examples of essential pairs: 
\begin{enumerate}
\item[(1)] Any special pair in the bypiramid $BP_n$ is essential (see Example \ref{ex3}). The $z$-monodromies of such pairs are from the class $K_i$ with $i\in\{0,1,2,3\}$. 
\item[(2)] The triangulation $\Gamma_{2,3,4,5}$ (Example \ref{ex4}) contains the following essential pairs: 
$v_0b,bv_1$ (the class $K_4$), $av_0,v_0b$ (the class $K_5$), $bv_6,v_6v_7$ (the class $K_{11}$), $bv_1,v_1v_2$ (the class $K_{12}$). 
\item[(3)] The triangulation $\Gamma_{2,4,3,4}$ (Example \ref{ex5}) contains the following essential pairs: 
$v_0b,bv_1$ (the class $K_6$), $bv_1,v_1v_2$ (the class $K_9$), $v_1v_2,v_2v_3$ (the class $K_{10}$). 
\end{enumerate}
}\end{exmp}

\begin{lemma}\label{lem4}
Let $g:\partial F_P\to\partial F_{P'}$ be a special homeomorphism. 
The number of zigzgas (up to reversing) in $\Gamma\#_{g}\Gamma'$ passing through the edges of $g(\omega(F_P))=\omega(F_{P'})$ is equal to the number of cycles in $g^{-1}M_{P'}gM_P$.
\end{lemma}

\begin{proof}
The proof is similar to the proof of Lemma \ref{lem2}. 
Consider the cyclic sequence
$$[x,M_{P}(x)], [gM_{P}(x),M_{P'}gM_{P}(x)], [g^{-1}M_{P'}gM_{P}(x),M_{P}(g^{-1}M_{P'}gM_{P})(x)],$$ 
$$\dots$$
$$[gM_P(g^{-1}M_{P'}gM_P)^{m-1}(x), M_{P'}gM_P(g^{-1}M_{P'}gM_P)^{m-1}(x)],$$
where $m$ is the smallest positive number such that $(g^{-1}M_{P'}gM_P)^{m}(x)=x$. 
We identify the last edge of any part in the above sequence with the first edge of the next part and obtain a zigzag in $\Gamma\#_{g}\Gamma'$. 
In this way, we get a one-to-one correspondence between zigzags passing through edges $g(\omega(F_P))=\omega(F_{P'})$ and cycles in the permutation $g^{-1}M_{P'}gM_P$. 
\end{proof}

Now, we prove Theorem \ref{th3}. 
Suppose that the $z$-orientation of $\Gamma$ contains more than one zigzag. 
Then there is a face $F$ such that at least two zigzags from the $z$-orientation pass through the edges of $F$. 
One of these zigzags contains the edge of type II on $F$. 
Denote this edge by $e$. 
There are two possibilities:

$\bullet$ $e$ occurs in two distinct zigzags;

$\bullet$ $e$ occurs twice in one zigzag $Z$.

\noindent We assert that for each of these cases there is a special pair $P$ contained in more than one zigzag from the $z$-orientation. 
The first case is obvious. 
In the second case, a zigzag distinct from $Z$ (we assume that this zigzag is from the $z$-orientation) passes through the edges of type I on $F$; the next edge $e'$ in this zigzag is on a face adjacent to $F$. 
By $z$-homogenity, $e'$ is of type II and it is easy to see that $e'$ has a common vertex with $e$. 
So, $e$ and $e'$ form a special pair. 

By Lemma \ref{lem3}, the $z$-monodromy $M_P$ belongs to one of the classes $K_i$ with $i\in\{0, 2, 4, 5, 6, 9, 10, 11, 12\}$. 
From this moment, we assume that $\Gamma'$ is a spherical triangulation and $P'$ is an essential pair. 
If $g:\partial F_P\to\partial F_{P'}$ is a special homeomorphism such that $g^{-1}M_{P'}gM_P$ is a $4$-cycle, then (by Lemma \ref{lem4}) the number of zigzags in $\Gamma\#_{g}\Gamma'$ is less than the number of zigzags in $\Gamma$. 
We show that for each of the nine possibilities of $M_P$ mentioned above there are suitable $\Gamma'$ and $P'$. 

Recall that $\omega(F)=\{1,2,3,4\}$ and assume that $\omega(F')=\{1',2',3',4'\}$. 
We will use the $z$-monodromies on spherical triangulations associated with essential pairs mentioned in Example \ref{ex6}. 
All gluings will be provided through the special homeomorphism $g:\partial F_P\to\partial F_{P'}$ transferring every $k\in\omega(F_P)$ to $k'$. 

If $K_i$ contains more than one permutation and $p\in K_i$, then the remaining elements of $K_i$ can be obtained from $p$ by the renumerations of the edges in the $4$-gonal face $F_P$ and the change of the $z$-orientation by the opposite. 
Therefore, if $K_i$ contains more than one permutation, then it is sufficiently to find suitable $\Gamma'$ and $P'$ only for one permutation from $K_i$. 

($K_0$). If $M_P=id$, then we define $M_{P'}$ as the permutation from $K_9$ or $K_{10}$ (see Example \ref{ex6}(3)). 
Since the permutations from $K_9$ and $K_{10}$ are $4$-cycles, the composition 
$$g^{-1}M_{P'}gM_P=g^{-1}M_{P'}g$$
also is a $4$-cycle. 

($K_2$). If  $M_P=(13)(24)$, then we take $M_{P'}=(1'2'3'4')\in K_1$ (see Example \ref{ex6}(1)) and 
$$g^{-1}M_{P'}gM_P=(1234)(13)(24)=(1432)$$
is a $4$-cycle. 

($K_4$). If  $M_P=(14)(23)$, then we take $M_{P'}=(2'4')\in K_6$ (see Example \ref{ex6}(3)) and 
$$g^{-1}M_{P'}gM_P=(24)(14)(23)=(1234)$$
is a $4$-cycle. 

($K_5$). If  $M_P=(12)(34)$, then we take $M_{P'}=(2'4')\in K_6$ (see Example \ref{ex6}(3)) and 
$$g^{-1}M_{P'}gM_P=(24)(12)(34)=(1432)$$
is a $4$-cycle. 

($K_6$). If $M_P$ belongs to $K_6$, then we take $M_{P'}$ from $K_4$ or $K_5$ (see Example \ref{ex6}(2)) and come to the cases ($K_4$) and ($K_5$). 

($K_9$). If $M_P$ belongs to $K_9$, then we take $M_{P'}$ from $K_0$ (see Example \ref{ex6}(1)) and come to the case ($K_0$). 

($K_{10}$). If $M_P$ belongs to $K_{10}$, then we take $M_{P'}$ from $K_0$ (see Example \ref{ex6}(1)) and come to the case ($K_0$). 

($K_{11}$). If  $M_P=(234)$, then we take $M_{P'}=(1'2'3'4')\in K_1$ (see Example \ref{ex6}(1)) and 
$$g^{-1}M_{P'}gM_P=(1234)(234)=(1243)$$
is a $4$-cycle. 

($K_{12}$). If  $M_P=(143)$, then we take $M_{P'}=(2'4')\in K_6$ (see Example \ref{ex6}(3)) and 
$$g^{-1}M_{P'}gM_P=(24)(143)=(1243)$$
is a $4$-cycle. 

So, any $z$-homogeneous triangulation of the surface $M$ can be modified to other $z$-homogeneous triangulation of $M$ with a lesser number of zigzags. 
Recursively, we come to a $z$-homogeneous triangulation of $M$ which is also $z$-knotted. 

\section{Appendix}

We present all necessary calculations for elements of $S_4$ used in Section 4. 
\begin{center}
\begin{tabular}{|C|C|C|C|C|C|C|C|}
\hline
\emph{$M_P$} & \emph{$M_P^{-1}$} & \emph{$sM_P$} & $\emph{$tM_P^{-1}$}$ & $\emph{$stM_P^{-1}$}$ & $\emph{$sM_Ps$}$ & $\emph{$tM_P^{-1}t$}$ & $\emph{$stM_P^{-1}st$}$ \\ \hline
$\emph{id}$ & $\emph{id}$ & $(13)(24)$ & $(12)(34)$ & $(14)(23)$ & $\emph{id}$ & $\emph{id}$ & $\emph{id}$ \\ \hline 
$(34)$ & $(34)$ & $(1324)$ & $(12)$ & $(1423)$ & $(12)$ & $(34)$ & $(12)$ \\ \hline 
$(23)$ & $(23)$ & $(1342)$ & $(1243)$ & $(14)$ & $(14)$ & $(14)$ & $(23)$ \\ \hline 
$(234)$ & $(243)$ & $(132)$ & $(123)$ & $(142)$ & $(124)$ & $(134)$ & $(123)$ \\ \hline 
$(243)$ & $(234)$ & $(134)$ & $(124)$ & $(143)$ & $(142)$ & $(143)$ & $(132)$ \\ \hline 
$(24)$ & $(24)$ & $(13)$ & $(1234)$ & $(1432)$ & $(24)$ & $(13)$ & $(13)$ \\ \hline 
$(12)$ & $(12)$ & $(1423)$ & $(34)$ & $(1324)$ & $(34)$ & $(12)$ & $(34)$ \\ \hline 
$(12)(34)$ & $(12)(34)$ & $(14)(23)$ & $\emph{id}$ & $(13)(24)$ & $(12)(34)$ & $(12)(34)$ & $(12)(34)$ \\ \hline 
$(123)$ & $(132)$ & $(142)$ & $(143)$ & $(124)$ & $(134)$ & $(124)$ & $(234)$ \\ \hline 
$(1234)$ & $(1432)$ & $(1432)$ & $(13)$ & $(24)$ & $(1234)$ & $(1234)$ & $(1234)$ \\ \hline 
$(1243)$ & $(1342)$ & $(14)$ & $(14)$ & $(1243)$ & $(1342)$ & $(1243)$ & $(1342)$ \\ \hline 
$(124)$ & $(142)$ & $(143)$ & $(134)$ & $(243)$ & $(234)$ & $(123)$ & $(134)$ \\ \hline 
$(132)$ & $(123)$ & $(234)$ & $(243)$ & $(134)$ & $(143)$ & $(142)$ & $(243)$ \\ \hline 
$(1342)$ & $(1243)$ & $(23)$ & $(23)$ & $(1342)$ & $(1243)$ & $(1342)$ & $(1243)$ \\ \hline 
$(13)$ & $(13)$ & $(24)$ & $(1432)$ & $(1234)$ & $(13)$ & $(24)$ & $(24)$ \\ \hline 
$(134)$ & $(143)$ & $(243)$ & $(132)$ & $(234)$ & $(123)$ & $(234)$ & $(124)$ \\ \hline 
$(13)(24)$ & $(13)(24)$ & $\emph{id}$ & $(14)(23)$ & $(12)(34)$ & $(13)(24)$ & $(13)(24)$ & $(13)(24)$ \\ \hline 
$(1324)$ & $(1423)$ & $(34)$ & $(1324)$ & $(34)$ & $(1423)$ & $(1423)$ & $(1324)$ \\ \hline 
$(1432)$ & $(1234)$ & $(1234)$ & $(24)$ & $(13)$ & $(1432)$ & $(1432)$ & $(1432)$ \\ \hline 
$(142)$ & $(124)$ & $(123)$ & $(234)$ & $(132)$ & $(243)$ & $(132)$ & $(143)$ \\ \hline 
$(143)$ & $(134)$ & $(124)$ & $(142)$ & $(123)$ & $(132)$ & $(243)$ & $(142)$ \\ \hline 
$(14)$ & $(14)$ & $(1243)$ & $(1342)$ & $(23)$ & $(23)$ & $(23)$ & $(14)$ \\ \hline 
$(1423)$ & $(1324)$ & $(12)$ & $(1423)$ & $(12)$ & $(1324)$ & $(1324)$ & $(1423)$ \\ \hline 
$(14)(23)$ & $(14)(23)$ & $(12)(34)$ & $(13)(24)$ & $\emph{id}$ & $(14)(23)$ & $(14)(23)$ & $(14)(23)$ \\ \hline 
\end{tabular}
\vspace{0.1cm}
\captionof{table}{ }
\end{center}

\end{document}